\documentclass{amsart}
\usepackage{amsmath,amsthm,bm,amssymb,extarrows,framed,enumerate,diagbox,mathtools,empheq}
\newtheorem{thm}{Theorem}[section]
\newtheorem{lem}[thm]{Lemma}
\newtheorem{cor}[thm]{Corollary}

\theoremstyle{dfn}

\theoremstyle{rmk}

\numberwithin{equation}{section}

\newcommand{\fr}{\displaystyle \frac}

\newcommand{\rd}{{\rm d}}

\newcommand{\supp}{{\rm supp}}
\newcommand{\bs}{\backslash}
\newcommand{\lam}{\lambda}
\newcommand{\kap}{\kappa}
\newcommand{\sig}{\sigma}

\newcommand{\R}{\mathbb{R}}

\newcommand{\calL}{\mathcal{L}}

\newcommand{\ov}{\overline}

\newcommand{\rar}{\rightarrow}
\newcommand{\lef}{\left}
\newcommand{\rig}{\right}

\newcommand{\rn}{{\R^N}}

\newcommand{\rnrn}{{\R^N\times\R^N}}

\newcommand{\kak}[1]{\lef(#1\rig)}
\newcommand{\chu}[1]{\lef\{#1\rig\}}
\newcommand{\dai}[1]{\lef[#1\rig]}

\newcommand{\norm}[1]{\lef|\lef|#1\rig|\rig|}
\newcommand{\norml}[2]{\lef|\lef|#1\rig|\rig|_{L^{#2}}}

\newcommand{\thmref}[1]{Theorem $\ref{#1}$}

\newcommand{\cirnum}[1]{\raise0.15ex\hbox{\rm \text{\textcircled{\raise-0.1ex\hbox{\footnotesize#1}}}}}
\newcommand{\cdotsnum}[1]{\cdots \raise-0.1ex\hbox{\cirnum{#1}}}
\newcommand{\cdotssen}[1]{\raise0.25ex\hbox{$\cdotsnum{#1}$}}

\begin{document}

\title{Remark on the formula for $L^p$ norm \\
characterized by weak $L^p$ norm}

\author{Yuto Miyadera}
\address{Department of Mathematics, Faculty of Science, Saitama University, Saitama 338-8570, Japan}
\email{y.miyadera.461@ms.saitama-u.ac.jp}



\date{June 2, 2023.}



\begin{abstract}
In \cite{s=0}, the new formula for $L^p(\R^N)$ norms was given by Gu and Yung which is
characterized by the standard weak-$L^p(\R^{2N})$ norms.  
At first glance, it seems to be able to measure the degree of fractional smoothness
of functions. In this paper, our proof goes in parallel as in \cite{s=0},
but gives another different representation which is obviously independent of
the smoothness of function.
\end{abstract}

\maketitle



\section{Introduction}
We start with the following fractional Sobolev norm, it is called Gagliardo
seminorm in \cite{s=1}, that is for $N\ge1$
\begin{align}\label{gag-norm}
\|u\|^p_{\dot{W}^{s,p}(\R^N)}=\iint_{\R^N\times\R^N}\frac{|u(x)-u(y)|^p}{|x-y|^{N+sp}}dxdy
\end{align}
where $1\le p<\infty,\ 0<s<1$. The expression of this integrant tempts us to expect $u(x)\to u(y)$ as $x\to y$ in 
some sense. For example, in the case $p=2$, we have the well-known equivalent norm, that is the 
homogeneous Sobolev norm $\dot{H}^s$, for $0<s<1$
\begin{align*}
\|u\|_{\dot{W}^{s,2}(\R^N)}^2\sim\|u\|_{\dot{H}^s(\R^N)}^2=\int_{\R^N}|\xi|^{2s}|\hat{u}(\xi)|^2d\xi
\end{align*}
where the hat is the Fourier transform $\hat{u}(\xi)=\int_{\R^N}e^{-ix\xi}u(x)dx$, and so the 
index $s$ corresponds to the number of derivatives. Unfortunately, this equivalence formula does not 
arrow the end-point cases $s=0,1$. This is precisely the most important point in this paper.
We also introduce one another norm. For $0\le\gamma<1$ this is a norm for the homogeneous
H\"older space for degree $\gamma$ as follows
\begin{align*}
\|u\|_{\dot\Lambda^{\gamma}}
=\sup_{x,y\in\R^N,x\not=y}
\frac{u(x)-u(y)}{|x-y|^{\gamma}}.
\end{align*}
A function belonging to this space is called H\"older continuous function of degree $\gamma$.  
Even though the norm for $\dot\Lambda^\gamma$ will not appear again after this, anyway, 
these are examples of norms that measure the smoothness of a function. \\

We introduce the results shown by Gu and Yung \cite{s=0} as follows. 
\begin{thm}[Qingsond Gu and Po-Lam Yung 2021]\label{thmGY}
For every $N\geq1$,\ there exists $c_1,c_2>0$ such that for all $1\leq p<\infty$ and all $u\in L^p(\rn)$, 
the following two inequalities hold,
\begin{align}\label{s=0}
c_1\norm{u}_{L^{p}(\R^N)}^p\leq&\dai{\fr{u(x)-u(y)}{|x-y|^{\frac{N}{p}}}}^p_{L^{p,\infty}(\rnrn)}
\leq 2^pc_2\norm{u}_{L^p(\R^N)}^p
\end{align}
where $L^{p,\infty}$ is the weak-$L^p$ space equipped by the quasi-norm as 
\begin{align*}
[f]_{L^{p,\infty}(\rnrn)}
=\sup_{\lambda>0}\kak{\lambda^p\calL^{2N}\kak{\chu{(x,y)\in\rnrn:|f(x,y)|\geq\lambda}}}^\frac{1}{p}.
\end{align*}
Moreover for the set
\begin{align*}
E_{\lambda}:=\chu{(x,y)\in\rnrn:|u(x)-u(y)|\geq\lambda|x-y|^{\frac{N}{p}}}, 
\end{align*}
the following limit holds 
\begin{align}\label{GY2}
\lim_{\lambda\to0+}\lambda^p\calL^{2N}(E_{\lambda})=2\kappa_N\|u\|_{L^p(\R^N)}^p
\end{align}
where $\kappa_N$ is the volume of the unit ball in $\R^N$. 
\end{thm}
Theorem \ref{thmGY} was shown in \cite{s=0}. 
In the middle term in \eqref{s=0}, 
we can find the same integrant in the right-hand side in \eqref{gag-norm} with $s=0$, 
although it is estimated in the level set instead of taking integration. 
The result \eqref{s=0} holds for any $u\in L^p(\R^N)$ and which gives the equivalency 
between the middle term characterized by weak $L^p(\R^N\times\R^N)$ and the standard 
$L^p(\R^N)$ norm. Since $L^p$ space does not require those functions with any smoothness, 
the weight function $|x-y|^{N/p}$ in the denominator should be harmless. 
Here we remark that the function $|x|^{-N/p}$ on $x\in\R^N$ is a typical example for 
which it belongs to $L^{p,\infty}(\R^N)$, not to $L^p(\R^N)$. 
We remark one more that, historically speaking, the result by Gu and Yung is a natural continuation of the work 
of the surprising formula for the homogeneous 
Sobolev norms with first-order derivative, that is the case $s=1$ and using level sets, 
by Brezis, Van Schaftingen, and Yung \cite{s=1}. And before them, there are results on asymptotic 
research on $\dot{W}^{s,p}$ as $s\to1$ and $s\to0$ in \cite{asymp1} and \cite{asymp0} respectively. 
They estimate the integral of the norm \eqref{gag-norm} itself, did not use level sets. 
Especially the result on $s\to1$ in \cite{asymp1} is called BBM theorem and which is considered the genesis of these studies.

Now we state our main result.
\begin{thm}\label{thm1}
For every $N\geq1$, there exist constants $c_1,c_2>0$ such that for all $1\leq p<\infty$ and all $u, v\in L^p(\rn)$,
\begin{align}\label{goal1}
\fr{c_1}{2^{p}}\kak{\|u\|_{L^p(\R^N)}+\|v\|_{L^p(\R^N)}}^p 
&\leq\left[\fr{u(x)+v(y)}{|x-y|^\frac{N}{p}}\right]_{L^{p,\infty}(\R^N\times\R^N)}^p \\
&\leq c_2\kak{\|u\|_{L^P(\R^N)}+\|v\|_{L^p(\R^N)}}^p. \label{goal2}
\end{align}
Moreover for the set
\begin{align}
E_\lambda=\chu{(x,y)\in\rnrn:|u(x)+v(y)|\ge\lambda|x-y|^\frac{N}{p}}. \label{*1}
\end{align}
then the following limit holds 
\begin{align}\label{limit0}
\lim_{\lambda\to0+}\lambda^p\calL^{2N}(E_{\lambda})
=\kappa_N\left(\|u\|_{L^p(\R^N)}^p+\|v\|_{L^p(\R^N)}^p\right). 
\end{align}
\end{thm}
We remark for this theorem that the two functions $u$ and $v$ are independent of each other. 
However, we may set some relationship between $u$ and $v$. 
As a corollary of \thmref{thm1}, we may derive Theorem \ref{thmGY} by setting $u=u, v=-u$. 
We also have the other formulas for the equivalency with $L^p$ norm. We take $u=|u|, v=\pm|u|$ 
and use the monotone property of $L^{p,\infty}$ norm in its integrants, see Lemma \ref{Lp-lem} below, 
to obtain the following. 
\begin{cor}\label{p2}
For every $N\geq1$, there exist constants $c_1,c_2>0$ such that for all $1\leq p<\infty$ and all $u\in L^p(\rn)$,
\begin{align}
c_1\|u\|_{L^p(\R^N)}^p&\leq\left[\fr{|u(x)|-|u(y)|}{|x-y|^\frac{N}{p}}\right]_{L^{p,\infty}(\R^N\times\R^N)}^p \\
&\leq\left[\fr{|u(x)|+|u(y)|}{|x-y|^\frac{N}{p}}\right]_{L^{p,\infty}(\R^N\times\R^N)}^p 
\leq2^pc_2\norm{u}_{L^p(\R^N)}^p.\label{2.1}
\end{align}
\end{cor}

\ \\

Here we give one investigation, that is the case of $u_2=0$ in Theorem \ref{thm1}.
The calculation is easy but this is the heart of the series of discussions in the papers \cite{s=0} and this paper. 
If we take $u_1=u, u_2=0$, then the set \eqref{*1} is the all  $(x,y)$ with $|u(x)|\ge\lambda|x-y|^{\frac{N}{p}}$ 
which is the ball centered at $x\in\R^N$ and radius $(|u(x)|/\lambda)^{p/N}$ for $y$. 
Hence,
\begin{align*}
{\lam^p}\calL^{2N}\kak{E_\lam}=&{\lam^p}\int_\rn\calL^{N}\kak{\ov{B\kak{x,\kak{\fr{|u(x)|}{\lam}}^\frac{p}{N}}}}\rd x\\
=&\int_\rn{\kap_N|u(x)|^p}\rd x=\kap_N\norm{u}_{L^p}^p,
\end{align*}
where $\kap_N$ is denoted as $\calL^N\kak{B(0,1)}$. Taking supremum in $\lambda$, and also by the symmetry 
between $x$ and $y$, we have
\begin{align}\label{heart}
\left[\fr{u(x)}{|x-y|^\frac{N}{p}}\right]_{L^{p,\infty}(\R^N\times\R^N)}^p
=\left[\fr{u(y)}{|x-y|^\frac{N}{p}}\right]_{L^{p,\infty}(\R^N\times\R^N)}^p 
=\kap_N\norm{u}_{L^p(\R^N)}^p. 
\end{align}

\ \\

In section 2 we will give a proof of Theorem \ref{thm1}. We follow the almost same argument 
with the proof of Theorem \ref{thmGY} made by Gu and Yung. 
Our contribution in this paper is giving some short cutting the argument by using 
known estimates, and more, reducing the number of areas to estimate the level set which might 
make a proof simpler.

\section{Proof of \thmref{thm1}}
Before going to proving Theorem \ref{thm1}, we introduce some tools in order to estimate $L^{p,\infty}$ norm.
See for example the book by Grafakos \cite{G-book}. 
\begin{lem}\label{Lp-lem}
Let $N\ge1, 1\le p\le\infty$. Then
\begin{enumerate}
\item the following quasi-triangle inequality holds
\begin{align}\label{triangle}
\dai{f+g}_{L^{p,\infty}(\rn)}\le2^{p-1}\kak{\dai{f}_{L^{p,\infty}(\rn)}+\dai{g}_{L^{p,\infty}(\rn)}}.
\end{align}
\item two functions satisfying $|f(x)|\le|g(x)|$ almost everywhere $x\in\rn$ 
follow the monotonicity
\begin{align}\label{monotone}
\dai{f}_{L^{p,\infty}(\rn)}\le\dai{g}_{L^{p,\infty}(\rn)}.
\end{align} 
\end{enumerate}
\end{lem}

Now we prove our main theorem. 
\begin{proof}[Proof of Theorem \ref{thm1}]
In this proof, we abbreviate $L^p(\R^N)=L^p$ and $L^{p,\infty}(\R^N\times\R^N)=L^{p,\infty}$ when 
there is no confusion. We will prove \eqref{goal2} and \eqref{goal1} in turns. For \eqref{goal1}, 
we will prove \eqref{limit0} since the supremum is bigger than or equal to the limit with respect to $\lambda$. 
For \eqref{limit0}, we take two steps, both $u$ and $v$ are compactly supported functions as the first step, 
and other cases as the second step.  
These are the same strategy with the proof in \cite{s=0}, see also \cite{s=1}. 

\subsection{Proof for the upper bound.}
Here we prove \eqref{goal2}. 
We apply \eqref{triangle} and \eqref{heart} and complete the proof just as
\begin{align}
\left[\fr{u(x)+v(y)}{|x-y|^\frac{N}{p}}\right]_{L^{p,\infty}}
&\lesssim\left[\fr{u(x)}{|x-y|^\frac{N}{p}}\right]_{L^{p,\infty}}
+\left[\fr{v(y)}{|x-y|^\frac{N}{p}}\right]_{L^{p,\infty}}\notag\\
&\sim\|u\|_{L^p}+\|v\|_{L^p}. \label{*2}
\end{align}

\subsection{Proof for the limit as $\lambda\to0+$ under the case $u$ and $v$ are compactly supported.}
As a first step, we assume $u$ and $v$ are compactly supported and satisfy 
\begin{align}\label{support}
\supp \ u, \supp \ v&\subset B_R
\end{align}
with $R>0$ 
where $B_R=B_R^N=B^N(0,R)$ is a $N$ dimensional ball centered at the origin and radius $R$. 
We divide the set $E_{\lambda}$ into 4 parts mutually disjoint such as
\begin{align*}
E_\lam=&\kak{E_{\lam}\cap\kak{B_R\times B_R}}\cup\kak{E_{\lam}\cap\kak{B_R\times B_R^c}}
\cup\kak{E_\lam\cap\kak{B_R^c\times B_R}} \\
&\cup\kak{E_\lam\cap\kak{B_R^c\times B_R^c}}
=:E_{\lambda}^1\cup E_{\lambda}^2\cup E_{\lambda}^3\cup E_{\lambda}^4. 
\end{align*}
Here we remark that the division by diagonal $|x|>|y|$ and $|x|<|y|$ is not applied which was done in \cite{s=0}. 
We see that $E_{\lambda}^4=\emptyset$ for any $\lambda>0$ since there we have $u(x)=v(y)=0$. 
We estimate each term in R.H.S. of the following
\begin{align*}
{\calL}^{2N}\kak{E_\lam}={\calL}^{2N}\kak{E_\lam^1}+{\calL}^{2N}\kak{E_\lam^2}+{\calL}^{2N}\kak{E_\lam^3}. 
\end{align*} 
The area $E_{\lambda}^1$ is bounded and we estimate
\begin{align*}
{\calL}^{2N}\kak{E_\lam^1}\le{\calL}^{2N}\kak{B_R\times B_R}=\kap^2_NR^{2N}.
\end{align*}
To estimate $E_{\lambda}^2$, we define $E_{\lam,x}^2:=\chu{y\in B_R^c :(x,y)\in E_\lam}$ 
for $x\in B_R$. We may write
\begin{align*}
E_{\lam,x}^2=\chu{y\in B_R^c :|x-y|\leq\kak{\fr{|u(x)|}{\lam}}^\frac{p}{N}}, 
\end{align*}
and which is estimated by
\begin{align*}
\ov{B\kak{x,\kak{\fr{|u(x)|}{\lam}}^\frac{p}{N}}}\cap B_R^c
=E_{\lambda,x}^2\subset \ov{B\kak{x,\kak{\fr{|u(x)|}{\lam}}^\frac{p}{N}}}. 
\end{align*}
We integrate this in $x\in B_R$ to have 
\begin{align*}
\kap_N\frac{\norml{u}{p}^p}{\lambda^p}-\kap^2_NR^{2N}\leq\calL^{2N}\kak{E_{\lam}^2}
\leq\kap_N\frac{\norml{u}{p}^p}{\lambda^p}. 
\end{align*}
From symmetry, we have 
\begin{align*}
\kap_N\frac{\norml{v}{p}^p}{\lambda^p}-\kap^2_NR^{2N}\leq\calL^{2N}\kak{E_{\lam}^3}
\leq\kap_N\frac{\norml{v}{p}^p}{\lambda^p}. 
\end{align*}
We add up the lower bounds and the upper bounds respectively for $\calL^{2N}\kak{E_{\lambda}^j}$, 
$j=1,2,3$ and multiple by $\lambda^p$ to have
\begin{align}
\kap_N\kak{\norml{u}{p}^p+\norml{v}{p}^p}-2{\lam^p}\kap^2_NR^{2N} \label{comp-est1}
&\leq{\lam^p}\calL^{2N}\kak{E_{\lam}} \\
&\leq\kap_N\kak{\norml{u}{p}^p+\norml{v}{p}^p}+\lam^p\kap^2_NR^{2N}. \label{comp-est2}
\end{align}
Therefore, under the condition the functions $u$ and $v$ are compactly supported, 
by letting $\lam\rar0+$ we obtain \eqref{limit0}.

\subsection{Proof for the limit as $\lambda\to0+$ in generous case.}
In this subsection, we completely follow the argument Gu and Yung \cite{s=0}. 
We repeat the proof here since it is not so long and for readers' convenience.
We define
\begin{align*}
u_R:=\chi_{B_R}u, \qquad v_R:=\chi_{B_R}v
\end{align*}
where a cutoff function satisfies $\chi_A(x)=1, x\in A$ and $0, x\in A^c$ with $A\subset\R^N$. 
Then $u_R$ and $v_R$ satisfy the condition \eqref{support}. We also define functions on 
the exterior of ball $B_R^c$ such as 
\begin{align*}
u_E:=(1-\chi_{B_R})u, \qquad v_E:=(1-\chi_{B_R})v,
\end{align*}
we then have $u=u_R+u_E, v=v_R+v_E$. 
Let $\sig\in(0,1)$ and define
\begin{align}
A_1=&\chu{(x,y)\in\rnrn:x\neq y,\ \fr{|u_{R}(x)+v_{R}(y)|}{|x-y|^\frac{N}{p}}\geq(1-\sig)\lam}\label{h1},\\
A_2=&\chu{(x,y)\in\rnrn:x\neq y,\ \fr{|u_{E}(x)+v_{E}(y)|}{|x-y|^\frac{N}{p}}\geq\sig\lam}\label{h2}, \\
A_3=&\chu{(x,y)\in\rnrn:x\neq y,\ \fr{|u_{R}(x)+v_{R}(y)|}{|x-y|^\frac{N}{p}}\geq(1+\sig)\lam}\label{h3}.
\end{align}
For those, we have the inclusions as 
\begin{align*}
A_3\bs A_2\subset E_\lam\subset A_1\cup A_2, 
\end{align*}
and so the estimates
\begin{align}
\lam^p\calL^{2N}(A_3)-\lam^p\calL^{2N}(A_2)\leq&\lam^p\calL^{2N}(E_\lam)\notag\\
\leq&\lam^p\calL^{2N}(A_1)+\lam^p\calL^{2N}(A_2)\label{i}.
\end{align}
From \eqref{*2} which we have already proven, there exists $c_3>0$ such that
\begin{align}
\sig^p\lam^p\calL^{2N}(A_2)\leq c_3\kak{\norml{u_{E}}{p}+\norml{v_{E}}{p}}^p\label{j1}.
\end{align}
Since $u_{R}$ and $v_{R}$ are compactly supported, \eqref{comp-est1} and \eqref{comp-est2} gives respectively
\begin{align}
(1-\sig)^p\lam^p\calL^{2N}(A_1)\leq&\kap_N\kak{\norml{u_{R}}{p}^p+\norml{v_{R}}{p}^p}
+{(1-\sig)^p\lam^p}\kap^2_NR^{2N}
\label{j2},\\
(1+\sig)^p\lam^p\calL^{2N}(A_3)\geq&\kap_N\kak{\norml{u_{R}}{p}^p+\norml{v_{R}}{p}^p}
-2{(1+\sig)^p\lam^p}\kap^2_NR^{2N}\label{j3}.
\end{align}
All together of \eqref{i}, \eqref{j1}, \eqref{j2} and \eqref{j3} gives 
\begin{align*}
&\fr{\kap_N\kak{\norml{u_{R}}{p}^p+\norml{v_{R}}{p}^p}}{(1+\sig)^p}-2{\lam^p}\kap^2_NR^{2N}-\fr{c_3\kak{\norml{u_{E}}{p}+\norml{v_{E}}{p}}^p}{\sig^p} \\
&\leq\lam^p\calL^{2N}(E_\lam)\\
&\leq\fr{\kap_N\kak{\norml{u_{R}}{p}^p+\norml{v_{R}}{p}^p}}{(1-\sig)^p}+{\lam^p}\kap^2_NR^{2N}+\fr{c_3\kak{\norml{u_{E}}{p}+\norml{v_{E}}{p}}^p}{\sig^p}.
\end{align*}
We now apply $R=\lam^{-\frac{p}{4N}}$ and 
\begin{align*}
\sig=\fr{\sqrt{\norml{u_{E}}{p}^p+\norml{v_{E}}{p}^p}}{1+\sqrt{\norml{u_{E}}{p}^p+\norml{v_{E}}{p}^p}}
\end{align*}
and then letting $\lam\rar0+$ and so $R\to\infty,\sigma\to0+$, we have
\begin{align*}
\lim_{\lam\rar0+}\lam^p\calL^{2N}(E_\lam)=\kap_N\kak{\norml{u}{p}^p+\norml{v}{p}^p}
\end{align*}
as desired. 
\end{proof}
\bibliographystyle{amsplain}

\begin{thebibliography}{10}
\bibitem{asymp1} Jean Bourgain, Ha\"{i}m Brezis, and Petru Mironescu, 
\textit{Another look at Sobolev spaces},  Optimal control and partial differential equations, IOS, Amsterdam (2001) 439–455. 
\bibitem{s=1} Ha\"{i}m Brezis, Jean Van Schaftingen, and Po-Lam Yung, \textit{A surprising formula for Sobolev norms}, Proc. Natl. Acad. Sci. \textbf{118} (8) (2021), e2025254118.
\bibitem{G-book} Loukas Grafakos, Classical {F}ourier analysis, 3rd ed., Graduate Texts in Mathematics, vol.249, Springer, New York, (2014).
\bibitem{asymp0} Vladimir Maz'ya and Tatyana Shaposhnikova, \textit {On the {B}ourgain, {B}rezis, and {M}ironescu theorem
              concerning limiting embeddings of fractional {S}obolev spaces}, J. Funct. Anal., \textbf{195} (2002) 230--238. 
\bibitem{s=0} Qingsong Gu and Po-Lam Yung, \textit{A new formula for $L^p$ norm}, J. Funct. Anal., \textbf{281} (2021) 109075.
\end{thebibliography}

\end{document}